\documentclass[a4paper, 11pt]{amsart}

\usepackage{amssymb, amsmath, amscd, latexsym, mathrsfs}
\usepackage{amsmath}
\usepackage{amsfonts}
\usepackage{appendix}
\usepackage[cmtip,all]{xy}

\newtheorem{theorem}{Theorem}[section]

\newtheorem{thm}{Theorem}[section]
\newtheorem{lem}[thm]{Lemma}
\newtheorem{cor}[thm]{Corollary}
\newtheorem{prop}[thm]{Proposition}

\theoremstyle{definition}
\newtheorem{defi}[thm]{Definition}
\newtheorem{ex}[thm]{Example}

\theoremstyle{remark}
\newtheorem{rmk}[thm]{Remark}
\numberwithin{equation}{section} 

\newcommand{\holim}{\textup{holim}}

\newcommand{\emb}{\textup{emb}}

\newcommand{\op}{\textup{op}}

\newcommand{\N} { \mathbb{N} }  

\newcommand{\R} { \mathbb{R} }

\newcommand{\Tot}{\textup{Tot}}
\newcommand{\con}{\textup{con}}
\newcommand{\homeo}{\textup{homeo}}
\newcommand{\holink}{\textup{holink}}
\newcommand{\haut}{\textup{haut}}
\newcommand{\Fin}{\textup{Fin}}

\newcommand{\twosub}[2]{\begin{array}{cc}
\scriptstyle{#1} \\ [-1mm] \scriptstyle{#2} \end{array}}

\newcommand{\holimsub}[1]{\begin{array}[t]{cc} \textup{holim} \\ [-1mm]
\scriptstyle{#1} \end{array}}

\newcommand{\ctsholimsub}[1]{\begin{array}[t]{cc} \textup{ctsholim} \\ [-1mm]
\scriptstyle{#1} \end{array}}

\newcommand{\hocolimsub}[1]{\begin{array}[t]{cc} \textup{hocolim} \\
[-1.7mm] \scriptstyle{#1} \end{array}}

\newcommand{\minsub}[1]{\begin{array}[t]{cc} \textup{min} \\
[-1.2mm] \scriptstyle{#1} \end{array}}

\begin{document}

\title{Occupants in simplicial complexes}
\author{Steffen Tillmann}%
\address{Math.~Institut, Universit\"at M\"unster, 48149 M\"unster, Einsteinstrasse 62, Germany}%
\email{s${}_-$till05@uni-muenster.de} 


\begin{abstract} 

Let $M$ be a smooth manifold and $K\subset M$ be a simplicial complex of codimension at least 3. Functor calculus methods lead to a homotopical formula of $M\setminus K$ in terms of spaces $M\setminus T$ where $T$ is a finite subset of $K$. This is a generalization of the author's previous work with Michael Weiss \cite{occupants}
where the subset $K$ is assumed to be a smooth submanifold of $M$
and uses his generalization of manifold calculus adapted for simplicial complexes.
%
\end{abstract}
\maketitle

\tableofcontents\setcounter{tocdepth}{3}

\section{Introduction}

Let $M$ be a smooth manifold and let $K$ be a simplicial complex\footnote[1]{By simplicial complex we mean a geometric simplicial complex, i.e. the geometric realization of a combinatorial simplicial complex.} of codimension $\geq 3$. Throughout this paper we assume that $K$ is a subset of $M$ such that each (closed) simplex of $K$ is smoothly embedded in $M$. We would like to recover the homotopy type of $M\setminus K$ from the homotopy types of the spaces $M\setminus T$ where $T$ is a finite subset of $K$. The finite subset $T\subset K$ could be regarded as a finite set of occupants. \\

It turns out that it is possible to find such a homotopical formula, but only if we allow standard thickenings of the finite subsets $T\subset K$ and inclusions between them. We get an interesting poset regarded as a category - the configuration category $\con (K)$ of $K$. 
The objects of $\con(K)$ are pairs $(T,\rho)$ where $T$ is a finite subset of $K$ and $\rho: T\rightarrow (0,\infty)$ is a function which assigns to each element $t\in T$ the radius $\rho(t)$ of the corresponding thickening using a standard metric on $K$. These pairs have to fulfil certain conditions, e.g. the thickenings of the elements $t\in T$ are pairwise disjoint (for a precise definition, see section \ref{confcat}). For each object $(T,\rho)$ in $\con(K)$, we get a corresponding open subset $V_{K}(T,\rho) \subset K$ which is the disjoint union of the open balls of radius $\rho(t)$ about the points $t\in T$. 
We note that for each element $(T,\rho)$ of the configuration category, there is an inclusion
\[
M\setminus K \rightarrow M\setminus V_{K}(T,\rho)
\]
and thus a map from $M\setminus K$ in to the associated homotopy limit.
The following theorem is our main result:

\begin{theorem}\label{maintheorem}
If the codimension $\dim(M)-\dim(K)$ is at least $3$, the canonical map
\[
M\setminus K \rightarrow \holimsub{(T,\rho)\in \con(K)} M\setminus V_{K}(T,\rho)
\]
is a weak equivalence.
\end{theorem}

Theorem \ref{maintheorem} is an application of manifold calculus adapted for simplicial complexes as developed in \cite{manifoldcalculusadapted}.
In this paper the configuration category $\con(K)$ is a convenient replacement of
the category of special open subsets $\cup_{k} \mathcal{O}k(K)$ there.
Recall: The objects of $\cup_{k} \mathcal{O}k(K)$ are those open subsets $V$ of $K$ which have finitely many components and where each component of $V$ is stratified isotopy equivalent to the open star of some simplex $\sigma$ in $K$ (intersection of the open stars of the vertices of $\sigma$). Roughly speaking, a stratified isotopy equivalence is a simplexwise smooth isotopy equivalence. \\

As to be expected from manifold calculus, there is a stronger version of our main result with restricted cardinalities (see Theorem \ref{thetheorem}). More precisely, the map from $M\setminus K$ into the homotopy limit over the full subcategory of $\con (K)$ of the set with restricted cardinality is highly connected depending on that cardinality. \\

Using Theorem \ref{maintheorem}, we can prove an approximation theorem of the boundary of the manifold in some cases.
Let $M$ be a Riemannian manifold with boundary $\partial M$. We can recover the homotopy type of $\partial M$ from the homotopy types of the spaces $M\setminus T$ where $T$ is a finite subset of $M\setminus \partial M$. Again, we need to allow thickenings of the finite subsets $T$ and inclusions between them. Therefore, we consider the configuration category $\con(M\setminus \partial M)$ of the interior of $M$ (see section \ref{formulation2} for a precise definition). For each object $(T,\rho)$ in $\con(M\setminus \partial M)$, we have again a corresponding open subset $V_{M\setminus\partial M}(T,\rho)$ (using the Riemannian metric) which is the union of the open balls of radius $\rho(t)$ about the points $t\in T$. The inclusions
\[
\partial M \rightarrow M\setminus V_{M\setminus\partial M}(T,\rho)
\]
induce a map from $\partial M$ into the homotopy limit taken over the category $\con(M\setminus \partial M)$.
Assume now that $M$ is a \textit{nice neighborhood} of a compact simplicial complex $K \subset M$ as defined in Definition \ref{niceneighb}. In particular, this means that $K$ is a retract of $M$ weakly equivalent to it.

\begin{theorem}\label{mainvariation}
If the codimension $\dim(M)-\dim(K)$ is at least $3$, the canonical map 
\[
\partial M \rightarrow \holimsub{(T,\rho)\in \con(M\setminus\partial M)} M\setminus V_{M\setminus\partial M}(T,\rho)
\]
is a weak equivalence.
\end{theorem}

In this case we also have a stronger version with restricted cardinalities (see Theorem \ref{generalization}) and it generalizes one of the main results in \cite{occupants}.
In the absence of the calculus for simplicial complexes as developed in \cite{manifoldcalculusadapted},
there we had to assume the existence of a smooth disk fiber bundle $M\rightarrow L$ with fiber dimension $c\geq 3$ where $L$ is a closed smooth submanifold of $M$. This condition is a special case of our \textit{nice neighborhood} condition here (see Example \ref{weaken}).
%
%
Thus morally, a nice neighborhood should be regarded as a simplexwise fiber bundle $M\rightarrow K$ with nice fibers. \\

In an application we will study the following question: Let $M$ be a smooth manifold with boundary. It is well-known that the boundary $\partial M$ can be recovered as the homotopy link of the basepoint in $M / \partial M \cong (M\setminus\partial M)\cup{\infty}$. Therefore, it is allowed to say that there is an action of the homeomorphism group $\homeo (M\setminus\partial M)$ on the pair $(M,\partial M)$ by homotopy automorphisms, i.e. each homeomorphism of $M\setminus \partial M$ determines a homotopy automorphism of the pair $(M,\partial M)$. But it is also well-known that there is a canonical map of topological grouplike monoids (if an explanantion is needed, see \S \ref{aplli})
\[
\homeo(M\setminus\partial M) \hookrightarrow \haut_{N\Fin}(\con(M\setminus\partial M))  
\]
where $N\Fin$ is the nerve of the category of finite sets and maps between them and $\haut_{N\Fin}(\con(M\setminus\partial M))$ is the space of the homotopy automorphisms of $\con(M\setminus\partial M)$ over $N\Fin$.
In \cite{confcat} Weiss studies the question in what cases the action of $\homeo (M\setminus\partial M)$ on the pair $(M,\partial M)$ by homotopy automorphisms extends to an action of $\haut_{N\Fin}(\con(M\setminus\partial M))$ on the pair $(M,\partial M)$ by homotopy automorphisms. This has also applications in \cite{pontryjagin}. We can generalize his result (see Theorem \ref{apl}): The action can be extended if the condition in Theorem \ref{mainvariation} is satisfied. \\

Our paper with Weiss \cite{occupants} attracted attention in applied topology because of possible relevance in the study of sensor network problems (for an introduction from the topological point of view see \cite{sensors}). The theory developed in this paper should allow the study of more general cases where the domain covered by sensors no longer needs to be a smooth manifold but is allowed to be a simplicial complex. More concretely, let $X$ be a subspace of a euclidean space. Assume we have a collection of points in $X$, each point equipped with a sensor. Each sensor covers a neighborhood of its location, for simplicity a ball of fixed radius. We want to recover the sensor region which is the space covered by the union of all sensors - often equal to $X$. If the location of the points is known, the problem becomes trivial. But if the location is not known (e.g. if the sensors are influenced by environmental factors), topological approaches are needed and the content of this paper, in particular, seems to be relevant to the study of such sensor network problems. \\



\textbf{Outline:} In section \ref{mcafsc} we recall the basic results of manifold calculus adapted for simplicial complexes. Using Goodwillie's homotopy functor calculus, we give general criteria for when a functor is analytic, resp. polynomial, and manifold calculus can be applied.   \\
In section 3 we will introduce the configuration categories of a simplicial complex and a smooth manifold. The configuration category carries a continuous structure. We will take this into account when we define homotopy limits. This leads to the notion of the \textit{continuous homotopy limit}. We prove that in cases important to us it is weakly equivalent to the ordinary (or discrete) homotopy limit. \\
In section 4 we will formulate Theorem \ref{maintheorem} more precisely as well as the stronger version with restricted cardinalities and compare it with the situation in \cite{occupants} where $K$ is replaced by a smooth submanifold. Then we use manifold calculus (adapted for simplicial complexes) to prove it.  \\
In section \ref{boundaryrecovered} we will define a \textit{nice neighborhood} of a simplicial complex embedded in a smooth manifold and explain how this is a generalization of a smooth disk bundle over a smooth manifold. We will prove Theorem \ref{mainvariation} and its stronger version with restricted cardinalities. In section 6 these results will be applied in our study of homotopy automorphisms of the pair $(M,\partial M)$. \\

\textbf{Notation:} The category $(Top)$ is the category of topological spaces. By a simplex $S$ of a simplicial complex, we mean a nondegenerate closed simplex. For such a simplex $S$, we denote by $\op(S)$ the open simplex. For a positive integer $k$, we set $[k]:=\left\{0,1,...,k\right\}$ and $\underline{k} := \left\{1,...,k\right\}$.  \\

\textbf{Acknowledgment:} This paper is a part of the author's PhD thesis under the supervision of Michael Weiss. It is a pleasure to thank him for suggesting this interesting topic and for his unfailing support.

\section{Manifold calculus adapted for simplicial complexes}\label{mcafsc}

In \cite{manifoldcalculusadapted} we develop a generalization of manifold calculus where the smooth manifold is replaced by a simplicial complex. The main results of this paper are applications of this theory. Therefore, we introduce the constructions and main results of \cite{manifoldcalculusadapted} and compare them with the homotopy functor calculus. The comparison leads to criteria which help us to apply manifold calculus (adapted to simplicial complexes).

\subsection{Definitions and main results}
All the constructions and results can be found in \cite{manifoldcalculusadapted}. We define the category $\mathcal{O}=\mathcal{O}(K)$ as follows: The objects are the open subsets of $K$ and the morphisms are inclusions, i.e. for $U,V\in\mathcal{O}$ there is exactly one morphism $U\rightarrow V$ if $U\subset V$ and there are no morphisms otherwise. 
\begin{defi}
Let $U,V\in\mathcal{O}$ be open subsets and let $f_{0},f_{1}: U \rightarrow V$ be two maps such that $f_{i}|_{U\cap S}$ is a smooth embedding from $U\cap S$ into $V\cap S$ for all simplices $S$ of $K$ and $i=0,1$. We call $f_{0}$ and $f_{1}$ \textit{stratified isotopic} if there is a continuous map
$H:U\times \left[ 0,1 \right] \rightarrow V$ such that 
\[
H|_{(U\cap S)\times \left[ 0,1 \right]}: (U\cap S)\times \left[ 0,1 \right]\rightarrow (V\cap S)
\] 
is a smooth isotopy from $f_{0}|_{U\cap S}$ to $f_{1}|_{U\cap S}$ for all simplices $S$ of $K$. In this case we call $H$ a \textit{stratified isotopy} (from $f_{0}$ to $f_{1}$). \\
Note: For an $n$-dimensional simplex $S$, we can regard $U\cap S$ as a subspace in the euclidean space $\R^{n+1}$.
\end{defi}
\begin{defi}\label{isotopyequ}
Let $U,V\in\mathcal{O}$ be two open subsets with $U\subset V$. The inclusion $i: U\rightarrow V$ is a \textit{stratified isotopy equivalence} if there is a map $e: V\rightarrow U$ such that 
$e|_{V\cap S}$ is an embedding from $V\cap S$ into $U\cap S$ for all simplices $S$ of $K$ and $i\circ e$, respectively $e\circ i$, is stratified isotopic to $id_{V}$, 
respectively $id_{U}$.
\end{defi}
\begin{defi}\label{good}
A contravariant functor $F: \mathcal{O}\rightarrow (Top)$ is \textit{good} if
\begin{enumerate}
\item $F$ takes stratified isotopy equivalences to weak homotopy equivalences
\item for every family $\left\{ V_{i} \right\}_{i\in \N}$ of objects in $\mathcal{O}$ with $V_{i}\subset V_{i+1}$ for all $i\in \N$, the following canonical map is a weak homotopy equivalence:
\begin{center}
$F(\cup_{i} V_{i}) \rightarrow \holim_{i}\hspace{1mm} F(V_{i})$
\end{center}
\end{enumerate}
\end{defi}

Recall: For a positive integer $k$, let $\mathcal{P}([k])$ be the power set of $[k]$. Then a functor from $\mathcal{P}([k])$ to topological spaces is a $(k+1)$-cube of spaces. 
\begin{defi}
Let $\chi$ be a cube of spaces. The \textit{total homotopy fiber} of $\chi$ is the homotopy fiber of the canonical map
\begin{align*}
\chi(\emptyset ) \rightarrow \holimsub{\emptyset\neq T\subset [k]} \hspace{0,15cm}\chi (T)
\end{align*}
If this map is a weak homotopy equivalence, we call the cube $\chi$ \textit{(weak homotopy) cartesian}.
\end{defi}
Now we define polynomial functors. To this end, let $F$ be a good functor, let $V\in \mathcal{O}$ be an open subset of $K$, and let $A_{0},A_{1},...,A_{k}$ be pairwise disjoint closed subsets of $V$ (for a positive integer $k$). Define a $k$-cube by
\begin{align}\label{polcube}
T\mapsto F(V\setminus \cup_{i\in T} A_{i})
\end{align}
\begin{defi}\label{polynomial}
The functor $F$ is \textit{polynomial of degree} $\leq k$ if the $k$-cube defined in (\ref{polcube}) is cartesian for all $V\in\mathcal{O}$ and pairwise disjoint closed subsets $A_{0},A_{1},...,A_{k}$ of $V$.
\end{defi}
\textbf{Notation}: 
Let $x\in K$ be given und let $\mathcal{S}_{x}$ be the open star of the open simplex containing $x$, i.e. $\mathcal{S}_{x} :=\cup_{S} \hspace{0,1cm} \op(S)$ where the union ranges over all closed simplices $S$ of $K$ such that $x$ is an element of $S$. 

\begin{defi}\label{specialopenset}
For a positive integer $k$, we define a full subcategory $\mathcal{O}k(K) = \mathcal{O}k$ of $\mathcal{O}$. Its objects are the open subsets $V\subset K$ with the following properties: $V$ has at most $k$ connected components and for each component $V_{0}$ of $V$, there is an $x\in K$ such that $V_{0}\subset \mathcal{S}_{x}$ and the inclusion $V_{0}\rightarrow\mathcal{S}_{x}$ is a stratified isotopy equivalence. An element of $\mathcal{O}k$ (for some $k$) is called a \textit{special open set}.
\end{defi}

\begin{thm}\label{F1F2}
Let $F_{1}\rightarrow F_{2}$ be a natural transformation between two $k$-polynomial functors. If $F_{1}(V)\rightarrow F_{2}(V)$ is a weak equivalence for all $V\in\mathcal{O}k$, it is a weak equivalence for all $V\in\mathcal{O}$. 
\end{thm}

Let $F: \mathcal{O} \rightarrow (Top)$ be a good functor. There is a concept of (relative) handle index in a simplicial complex \cite[\S 3.1]{manifoldcalculusadapted}. We can use it to define analyticity for $F$. 
To this end, let $P$ be a compact codimension zero subobject of $K$ and let $\rho$ be a fixed integer. Suppose $A_{0},A_{1},...,A_{r}$ are pairwise disjoint compact codimension zero subobjects of $K\setminus\text{int}(P)$ with relative handle index $q_{A_{i}}\leq \rho$ (relative to $P$). For $T\subset [r]$, we set $A_{T}:= \cup_{i\in T} A_{i}$ and assume $r\geq 1$.
\begin{defi}\label{nalytic}
The functor $F$ is called \textit{$\rho$-analytic} \textit{with excess $c$} if, in these circumstances, the cube 
\begin{align*}
T\mapsto F(\text{int} \left(P \cup  A_{T})\right) \hspace{0,1cm},\hspace{0,3cm} T\subset [r]
\end{align*}
is $c + \sum_{i=0}^{r} (\rho - q_{A_{i}})$-cartesian for some integer $c$.
\end{defi}

\begin{thm}\cite[3.6]{manifoldcalculusadapted}\label{analytic}
Let $F$ be a $\rho$-analytic functor with excess $c$ and let $V\in\mathcal{O}$ be an open subset. Then the map
\begin{align*}
\eta_{k-1}(V):F(V)\rightarrow T_{k-1}F(V)
\end{align*}
is $(c+k(\rho - \text{dim}(K)))$-connected for every $k>1$.
\end{thm}

\begin{rmk}
Theorem \ref{analytic} is weaker than \cite[Theorem 3.6]{manifoldcalculusadapted} which uses the homotopy dimension of $V$ \cite[Def. 3.4]{manifoldcalculusadapted} in order to increase the connectivity. For our purposes we do not need this stronger version.
\end{rmk}

\begin{cor}
Let $F$ be a $\rho$-analytic functor with $\rho > \text{dim}(K)$. For all open sets $V\in\mathcal{O}(K)$, the canonical map
\begin{align*}
F(V)\rightarrow T_{\infty}F(V) = \holim_{k} T_{k}F(V)
\end{align*}
is a weak equivalence.
\end{cor}

\subsection{Comparison with homotopy functor calculus}

In the last section we introduced a version of manifold calculus for simplicial complexes. We saw that in order to apply the approximation theorem (\ref{analytic}), we need to assume analyticity of the functor. Therefore, we should look for criteria which implies that a functor is analytic. Surprisingly, homotopy functor calculus introduced by Goodwillie \cite{goodwillie2} helps to find such criteria. \\
Functor calculus investigates (covariant) \textit{homotopy functors} from topological spaces to itself. A functor $G:(Top)\rightarrow (Top)$ is called homotopy functor if it takes weak equivalences to weak equivalences. If $G$ is such a functor, we can compose it with a contravariant functor $F$ from $\mathcal{O}(K)$ to $(Top)$. The composition $G\circ F$ is a contravariant functor from $\mathcal{O}(K)$ to $(Top)$. We will examine this composition. 

\begin{defi}
A cube of spaces is called \textit{strongly cocartesian} if each sub 2-face is a homotopy pushout. 
\end{defi}

\begin{defi}\label{kpolf}
A homotopy functor $G$ from $(Top)$ to itself is called \textit{polynomial of degree $\leq k$} if it takes any strongly cocartesian $(k+1)$-cube to a weakly cartesian $(k+1)$-cube.
\end{defi}

Let $V\in \mathcal{O}(K)$ be an open subset of $K$, let $A_{0},A_{1},...,A_{k}$ be pairwise disjoint closed subsets of $V$ (for a positive integer $k$) and let $A_{T}:=\cup_{i\in T}A_{i}$ where $T$ is a subset of $[k]$. The following proposition is an easy observation.

\begin{prop}
Let $F: \mathcal{O}(K)\rightarrow (Top)$ be a good (contravariant) functor (Def. \ref{good}) such that
\begin{equation*}
\begin{gathered}
\xymatrix{
F(V\setminus A_{ T\cap T^{'}}) \ar[d] \ar[r] & F(V\setminus A_{T}) \ar[d] \\
F(V\setminus A_{T^{'}}) \ar[r] & F(V\setminus A_{T\cup T^{'}})
}
\end{gathered}
\end{equation*}
is a homotopy pushout for all $T,T^{'}\subset [k]$ and all choices of $V, A_{0},...,A_{k}$ as above
and let $G:(Top)\rightarrow (Top)$ be a (covariant) homotopy functor. We suppose that $G$ is $k$-polynomial in the sense of homotopy functor calculus (Def. \ref{kpolf}). Then the composition $G\circ F$ is $k$-polynomial in the sense of manifold calculus (adapted for simplicial complexes).
\end{prop}

We would like to have a similar statement for analyticity.

\begin{defi}
Let $\rho$ be an integer and let $\chi$ be a cocartesian $k$-cube of spaces such that the maps $\chi (\emptyset) \rightarrow \chi (\left\{i\right\})$ are $k_{i}$-connected with $k_{i}>\rho$ for all $i\in [k]$.
A homotopy functor $G$ is called \textit{$\rho$-analytic} with excess $c$ if the cube $G\circ \chi$ is $(c + \sum_{i\in [k]} (k_{i}-\rho))$-cartesian (for all choices of $\chi$).
\end{defi}

\begin{ex}\label{exid}
According to the Blakers-Massey theorem \cite{blakersmassey}, for any strongly cocartesian cube $\chi$ where the map $\chi(\emptyset)\rightarrow \chi(\left\{i\right\})$ is $\kappa_{i}$-connected for each $i\in [k]$, the cube $\chi$ is $\kappa$-cartesian with $\kappa =1+\sum_{i\in [k]} (\kappa_{i}-1)$.
Therefore by definition, the identity functor $\text{id}:(Top)\rightarrow (Top)$ is $1$-analytic with excess $1$. 
\end{ex}

Let $F: \mathcal{O}(K) \rightarrow (Top)$ be a good functor (Def. \ref{good}). Recall that there is a concept of relative handle index in a simplicial complex \cite[\S 3.1]{manifoldcalculusadapted}. 
Let $P$ be a compact codimension zero subobject of $K$ and let $\rho$ be a fixed integer. Suppose $A_{0},A_{1},...,A_{r}$ are pairwise disjoint compact codimension zero subobjects of $K\setminus\text{int}(P)$ with relative handle index $q_{A_{i}}\leq \rho$ (relative to $P$). For $T\subset [k]$, we set $A_{T}:= \cup_{i\in T} A_{i}$ and assume $k\geq 1$.

\begin{prop}\label{vergleichmcfc}
Suppose that the cube 
\begin{align*}
T\mapsto F(\text{int} \left(P \cup  A_{T})\right) \hspace{0,1cm},\hspace{0,3cm} T\subset [k]
\end{align*}
is strongly cocartesian and suppose that there is a positive integer $\delta$ such that the maps
\[
F(\text{int} \left(P \cup  A_{[k]})\right) \rightarrow F(\text{int} \left(P \cup  A_{[k]\setminus \left\{i\right\}})\right)
\] 
are $(\delta-q_{A_{i}})$-connected. Then $F$ is $(\delta-1)$-analytic with excess $1$ (in the sense of Def. \ref{nalytic}).
\end{prop}
\begin{proof}
The idea is to apply the Blakers-Massey theorem. By assumption, the cube $T \mapsto F(\text{int} \left(P \cup  A_{T})\right)$ is strongly cocartesian. We consider the cube 
\begin{align*}
T\mapsto \text{id} \circ F(\text{int} \left(P \cup  A_{T})\right) \hspace{0,1cm},\hspace{0,3cm} T\subset [k]
\end{align*}
By applying example \ref{exid}, we deduce that the cube is $1 + \sum_{i\in [k]} (\delta - q_{A_{i}} -1)$-cartesian. 
\end{proof}

\begin{rmk}
In the last proposition we use the analyticity of the identity map in topological spaces to find a criteria for analyticity of $F$ where $F$ is a good functor. More generally, the following statement holds: For a $\rho$-analytic functor $G: (Top) \rightarrow (Top)$ with excess $c$ and $F$ as above, the composition $G\circ F$ is a $(\delta-\rho)$-analytic functor with excess $c$ and where $\delta$ is as above. \\
For an additional short note on the relationship of manifold calculus (for smooth manifolds) and homotopy functor calculus, see \cite[Remark 1.3.2]{occupants}.
\end{rmk}

\section{Background}

In this section we provide some background which we will need for the discussions in the next sections. We introduce the configuration category of a simplicial complex and the continuous homotopy limit.

\subsection{Configuration category of a simplicial complex}\label{confcat}

We will need the configuration category of a manifold as well as the configuration category of a simplicial complex. First, we recall the Riemannian model of the configuration category of a smooth manifold. Note that there are several equivalent definitions of the configuration category of a manifold \cite{pedromichael}. \\
Let $M$ be a smooth manifold without boundary of dimension $m$ and suppose that we fixed a Riemannian metric on $M$. Then the configuration category $\con(M)$ of $M$ is a topological poset. The objects are pairs $(T,\rho)$ where $T$ is a finite subset of $M$ and $\rho : T\rightarrow (0,\infty)$ is a function such that
  \begin{enumerate}
     \item for each $t\in T$, the exponential map $\text{exp}_{t}$ is defined and regular on the compact disk of radius $\rho (t)$ about the origin in the tangent space $T_{t}M$.
		 \item the images in $M$ of these disks under the exponential maps $\text{exp}_{t}$ are pairwise disjoint.
  \end{enumerate}
For such a pair $(T,\rho)$, let $V_{M}(T,\rho)\subset M$ be the union of the open balls of radius $\rho(t)$ about $t\in T$. Then $V_{M}(T,\rho)$ is an open subset of $M$ which is diffeomorphic to $T\times \R^{m}$. All these pairs form a topological poset $\con(M)$ by 
\[
(T,\rho) \leq (T',\rho ') \hspace{0,5cm} : \Longleftrightarrow \hspace{0,5cm} V_{M}(T,\rho) \subset V_{M}(T',\rho ')
\]
This poset can also be regarded as a category. 
We would like to adapt this definition and introduce the configuration category $\con (K)$ of the simplicial complex $K$. Therefore, we should start with the following observation.
\begin{rmk}
Let $x$ be an element of $K$ and let $\mathcal{S}_{x}$ be the open star neighborhood of $x$ in $K$. The closure $K_{x} := \text{cl}(\mathcal{S}_{x})$ of $\mathcal{S}_{x}$ in $K$ carries a canonical metric $d = d_{x}$ induced by the euclidean structure of each simplex. The precise definition is technical and can be done by distinguishing the following two cases: If two elements $y,y' \in K_{x}$ are in the same simplex, we can use the euclidean structure of the simplex to define $d(y,y')\in [0,\infty )$ as the distance of $y$ and $y'$ in the euclidean space. If they are not in the same simplex, we set
\[
d(y,y') := \minsub{z\in S_{y}\cap S_{y'}} d(y,z) + d(z,y')
\]
where $S_{y}$, resp. $S_{y'}$, is the simplex of maximal dimension which includes $y$, resp. $y'$. By definition, we can use again the euclidean structure. \\
Note: We wrote $d$ instead of $d_{x}$ to avoid the index $x$. In fact, $d(y,y')$ is independent of the element $x$ in $K$: If $x,x'$ are two elements of $K$ with $y,y'\in \mathcal{S}_{x} \cap\mathcal{S}_{x'}$, then $d_{x}(y,y') = d_{x'}(y,y')$.
\end{rmk}

Now we introduce the configuration category $\con(K)$. The objects are again pairs $(T,\rho)$ where $T$ is a finite subset of $K$ and $\rho : T\rightarrow (0,\infty)$ is a function fulfilling the following two conditions.
  \begin{enumerate}
	   \item For each $t\in T$, there is an element $x\in K$ such that $t\in \mathcal{S}_{x}$ and the open ball $B^{d}_{\rho (t)}(t) \subset K_{x} = \text{cl}(\mathcal{S}_{x})$ of radius $\rho (t)$ about $t$ determined by the metric $d = d_{x}$ is a subset of the open star neighborhood $\mathcal{S}_{x}$ and the inclusion $B^{d}_{\rho (t)}(t) \hookrightarrow\mathcal{S}_{x}$ is a stratified isotopy equivalence (Def. \ref{isotopyequ}). In particular $B^{d}_{\rho (t)}(t)\in \mathcal{O}1$ is a special open set (see Def. \ref{specialopenset}).
     \item The open balls $B^{d}_{\rho (t)}(t)\subset K$ with origin $t$ and radius $\rho (t)$ are pairwise disjoint.
  \end{enumerate}
For such a pair $(T,\rho)$, let $V_{K}(T,\rho)\subset M$ be the union of the open balls $B^{d}_{\rho (t)}(t)\subset K$ of radius $\rho(t)$ about $t\in T$. Then $V_{K}(T,\rho)$ is a special open subset of $K$ (Def. \ref{specialopenset}). By analogy with the manifold case, we form the topological poset $\con(K)$ by 
\[
(T,\rho) \leq (T',\rho ') \hspace{0,5cm} : \Longleftrightarrow \hspace{0,5cm} V_{K}(T,\rho) \subset V_{K}(T',\rho ')
\]
This poset can also be regarded as a category. \\
Now we want to take a closer look at the configuration category $\con(K)$. But note that the following results are also true for $\con(M)$ - the configuration category of a smooth manifold $M$ (without boundary).

\begin{rmk}\label{topcat}
The configuration category $\con(K)$ is a topological poset, i.e. each, the objects and the morphisms, form a topological space. More generally, if $N(\con (K))$ is the nerve of the category $\con(K)$, then $N_{r}(\con(K))$ is a topological space for all $r\geq 0$. This is obvious since $N_{r}(\con(K))$ is the space of all strings
\[
(T_{0},\rho_{0}) \leq (T_{1},\rho_{1}) \leq ... \leq (T_{r},\rho_{r})
\]
where $(T_{i},\rho_{i})$, $0\leq i\leq r$, is an element of $\con(K)$.
\end{rmk}

Now we want to investigate the homotopy type of the configuration category $\con(K)$ as topological space. It is very reminiscent of the configuration spaces.

\begin{defi}
We define $C_{r}(K)$ to be the space of unordered configurations of $r$ points in $K$: Let $F_{r}(K)$ be the space of ordered $r$-configurations of $K$ given by
\[
F_{r}(K) := \left\{ (x_{1},...,x_{r})\in K^{r} \mid x_{i}\neq x_{j} \hspace{0,2cm}\text{for all}\hspace{0,2cm} i\neq j \right\} 
\]
The symmetric group $\Sigma_{r}$ acts freely on $F_{r}(K)$. Then 
\[
C_{r}(K) := F_{r}(K) / \Sigma_{r}
\]
is the space of unordered $r$-configurations.
\end{defi}

\begin{rmk}
What is the relation between the configuration category and the configuration spaces? Let $r\geq 0$ be a fixed integer. We define the space $C_{r}^{\text{fat}}(K)$ to be the space of all pairs $(T,\rho)\in\con(K)$ with $\left|T\right| = r$. Then we have a forgetful projection map 
\[
C_{r}^{\text{fat}}(K) \rightarrow C_{r}(K)
\]
which is a fiberbundle with contractible fibers. Therefore, this map is a weak equivalence of spaces.
\end{rmk}

\subsection{Continuous homotopy limit}\label{conthomotlim}

Let $\con(K)$ be the configuration category of $K$ and let $N(\con (K))$ be its nerve. We saw that $N_{r}(\con(K))$ is a topological space for all $r\geq 0$. We are studying the functor $\Phi$ from $\con(K)$ to topological spaces defined by
\[
\Phi((T,\rho)) := M\setminus V_{K}(T,\rho)
\]
and its homotopy limit 
\[
\holimsub{\con(K)} \Phi \hspace{0,15cm} = \hspace{0,05cm} \holimsub{(T,\rho)\in\con(K)} M\setminus V_{K}(T,\rho)
\]
During our study of this homotopy limit, we would like to integrate the continuous structure of the nerve of $\con(K)$. To this end, we will introduce the continuous homotopy limit of $\Phi$ using the topological structure of the configuration category. \\

We recall that the ordinary (or discrete) homotopy limit $\holim_{\con(K)} \Phi $ of the contravariant functor $\Phi$ is defined to be the totalization of the cosimplicial space
\[
\left[r\right] \mapsto \prod_{(T_{0},\rho_{0})\leq ... \leq(T_{r},\rho_{r})\in N_{r}(\con(K))} \Phi((T_{r},\rho_{r}))
\]
By definition, the target is equal to the space of all (not necessarily continuous) sections from $N_{r}(\con(K))$ to 
\[
\coprod_{(T_{0},\rho_{0}) \leq ... \leq (T_{r},\rho_{r})\in N_{r}(\con(K))} \Phi((T_{r},\rho_{r}))
\]
Equivalently, the target is equal to the space of all (not necessarily continuous) maps $f : N_{r}(\con(K)) \rightarrow M$ such that 
\[
f((T_{0},\rho_{0}) \leq ... \leq (T_{r},\rho_{r})) \in M\setminus V_{K}(T_{r},\rho_{r})
\]
Using the continuous structure of $\con(K)$, we introduce the following notation.

\begin{defi}
We define $\Gamma_{r}(\Phi)$ as the space of all continuous maps $f : N_{r}\con(K)\rightarrow M$ such that $f((T_{0},\rho_{0})\leq...\leq(T_{r},\rho_{r})) \in M\setminus V_{K}(T_{r},\rho_{r})$. 
\end{defi}

Note: If we define $E^{!}_{r}(\Phi)$ to be the space 
\[
\coprod_{(T_{0},\rho_{0})\leq...\leq(T_{r},\rho_{r})\in N_{r}(\con(K))} \Phi((T_{r},\rho_{r})) 
\]
with the subspace topology of $N_{r}(\con(K)) \times M$, then the projection map $E^{!}_{r}(\Phi)\rightarrow N_{r}(\con(K))$ is a fiber bundle and $\Gamma_{r}(\Phi)$ is the space of all continuous sections of this fiber bundle.

\begin{defi}
The \textit{continuous homotopy limit} $\text{ctsholim}_{\con(K)}\Phi$ of $\Phi$ is defined to be the totalization of the cosimplicial space $\left[ r \right] \mapsto \Gamma_{r}(\Phi)$.
\end{defi}

\begin{lem}\label{holimcomp}
The canonical inclusion map
\[
\ctsholimsub{\con(K)}\Phi \rightarrow \holimsub{\con(K)}\Phi
\]
is a weak equivalence.
\end{lem}

We skip the proof because it is equal to the proof of \cite[Lemma 1.2.1]{occupants}. (If we replace the manifold $L$ appearing in \cite[1.2.1]{occupants} by the simplicial complex $K$, then we get a proof for Lemma \ref{holimcomp}.)

Using this result, we can work in the following with either of these homotopy limits - the discrete homotopy limit or the continuous homotopy limit.

\begin{rmk}\label{resholimcomp}
We will need the following observation. For an open subset $U$ of $K$, let $\con(K)|_{U}$ be the full subcategory of $\con(K)$ such that the objects are all elements $(T,\rho)$ in $\con(K)$ with $V_{K}(T,\rho) \subset U$. For $r\geq 0$, let $\Gamma_{r}(\Phi)|_{U}$ be the space of all continuous maps $f: N_{r}(\con(K)|_{U}) \rightarrow M$ such that
\[
f((T_{0},\rho_{0}) \leq ... \leq (T_{r},\rho_{r}))\in M\setminus V_{K}(T_{r},\rho_{r})
\]
Now we define $\text{ctsholim}_{\con(K)|_{U}} \Phi$ to be the totalization of the cosimplicial space $r\mapsto \Gamma_{r}(\Phi)|_{U}$. There is a canonical inclusion map
\[
\ctsholimsub{\con(K)|_{U}}\Phi \rightarrow \holimsub{\con(K)|_{U}}\Phi
\]
which is a weak equivalence. The proof is equal to that of Lemma \ref{holimcomp}.
\end{rmk}

\begin{rmk}
The cosimplicial space $r\mapsto \Gamma_{r}(\Phi)|_{U}$ is Reedy fibrant for every open subset $U$ of $K$. The verification is equal to that in \cite[1.1.3]{occupants}. Recall that for a map $X\rightarrow Y$ between cosimplicial spaces which is a degreewise weak equivalence, the map of their totalizations $\Tot (X)\rightarrow \Tot (Y)$ is a weak equivalence.
\end{rmk}

\section{The main theorem}

We formulate the main theorem and apply manifold calculus (adapted to simplicial complexes) in order to prove it.

\subsection{The formulation of the problem}\label{formulation}

We remind the reader that $M$ is a smooth manifold and $K\subset M$ is a simplicial complex such that each (closed) simplex of $K$ is smoothly embedded in $M$. For each element $(T,\rho)$ of the configuration category $\con(K)$, there is an inclusion map
\[
M \setminus K \rightarrow M \setminus V_{K}(T,\rho)
\]
where $V_{K}(T,\rho)$ is the open subset of $K$ corresponding to the pair $(T,\rho)$. If we define a contravariant functur $\Phi$ from $\con(K)$ to topological spaces by $\Phi ((T,\rho)) := M \setminus V_{K}(T,\rho)$, then the inclusion maps induce a canonical map
\begin{align}\label{canmap}
M \setminus K \rightarrow  \holimsub{\con(K)} \Phi
\end{align}
We can ask if the canonical map is a weak equivalence. 
There is a variant with restricted cardinalities. Let $n\geq 0$ be an integer. Then we define $\con_{\leq n}(K)$ to be the full subcategory of $\con(K)$ where the objects are all elements $(T,\rho)$ of $\con(K)$ with $\left|T\right| \leq n$. Again, we get a canonical map
\begin{align}\label{canmapres}
M \setminus K \rightarrow  \holimsub{\con_{\leq n}(K)} \Phi
\end{align}
induced by inclusions. In this case we do not expect that this map is a weak equivalence. But we can ask if it is highly connected. In the following theorem we use the notations $m := \dim (M)$ and $\kappa := \dim (K)$.

\begin{thm}\label{thetheorem}
If $\kappa+3\leq m$, then the canonical map (\ref{canmap}) is a weak equivalence and (\ref{canmapres}) is $1+(n+1)(m-\kappa-2)$-connected.
\end{thm}

\begin{rmk}
The homotopy limit appearing in (\ref{canmap}) is the ordinary (or discrete) homotopy limit. By Lemma \ref{holimcomp}, we could also use the continuous homotopy limit and the theorem would still hold. Using similar arguments, we could also use the continuous homotopy limit in (\ref{canmapres}).
\end{rmk}

\begin{rmk}
We assumed that the codimension of $M$ and $K$ is at least three. In fact, the theorem would be false without this assumption. There is a nice counterexample in codimension two \cite[Remark 1.3.3]{occupants}.
\end{rmk}

\begin{rmk}
The theorem is a generalization of \cite[Theorem 1.1.1]{occupants}. Let $L$ be a compact, smooth submanifold (without boundary) of $M$ where the codimension of $M$ and $L$ is at least three. We can choose a triangulation of $L$ and get a simplicial complex $K$, i.e. $K=L$ as topological space but the configuration categories $\con(L)$ and $\con(K)$ are not equal. Then there is a zigzag of weak equivalences
\[
\holimsub{(T,\rho)\in \con(L)} \Phi((T,\rho))\hspace{0,1cm} \leftarrow \holimsub{U \in \cup_{k}\mathcal{O}k(L)} M\setminus U \hspace{0,1cm}\rightarrow \holimsub{(T,\rho)\in \con(K)} \Phi((T,\rho))
\]
where $\cup_{k}\mathcal{O}k(L)$ is the category of all special open subsets of $L$ \cite{mancal1}. These are all open subsets of $L$ which are diffeomorphic to a disjoint union of open disks. Both maps are given by inclusion of categories. 
\end{rmk}

\subsection{A good functor}

In order to prove Theorem \ref{thetheorem}, we would like to apply manifold calculus (adapted to simplicial complexes). Naively, one could suggest to apply the approximation theorem (Theorem \ref{analytic}) to the contravariant  functor which maps an open subset $V\subset K$ to the topological space $M\setminus V$. Unfortunately, this functor is not good because in general it does not take stratified isotopy equivalences to weak equivalences (for a counterexample, see \cite[1.3]{occupants}). Therefore, we need a modification.

\begin{defi} 
We define the functor $F$ from the category $\mathcal{O}(K)$ of open subsets of $K$ to topological spaces by
\[
F(V) := \holimsub{C\subset V} M \setminus C
\]
where $C$ runs over all compact subsets of $V$.
\end{defi}

We will see that $F$ is an appropriate replacement of the functor $V\mapsto M\setminus V$.
The proof in the following lemma is similar to that of \cite[1.3.1]{occupants}. For the sake of completeness, we will give all required arguments. 

\begin{lem}\label{Fisgood}
The functor $F$ is good (in the sense of Definition \ref{good}).
\end{lem}
\begin{proof}
First, we notice that the (co-)limit axiom is fulfilled. This is obvious. In order to show that the functor takes stratified isotopy equivalences to weak homotopy equivalences, we will use the reformulation of stratified isotopy equivalences as given in Remark \ref{equivisot}. To this end, let $V_{0}$ and $V_{1}$ be two open subsets of $K$ with $V_{0}\subset V_{1}$ and let $e_{t}:V_{0}\rightarrow V_{1}$, $t\in \left[ 0,1 \right]$, be a stratified isotopy such that $e_{0}$ is the inclusion and for each simplex $S$ of $K$, $e_{1}$ is a homeomorphism such that $e_{1}|_{S}:S\cap V_{0}\rightarrow S\cap V_{1}$ can be extended to a diffeomorphism (see Remark \ref{equivisot}).  \\
Let $\left\{C_{i}\right\}_{i\geq 0}$ be a sequence of compact subsets of $V_{1}$ such that $C_{i}\subset C_{i+1}$ for all $i\geq 0$ and such that for every compact subset $C$ of $V_{1}$, there is an element $C_{i}$ of this sequence with $C\subset C_{i}$. By definition, the inclusion 
\[
\left\{C_{i}\right\}_{i\geq 0} \rightarrow \left\{C\subset V_{1} \mid C\hspace{0,2cm} \text{compact} \right\}
\]
is homotopy terminal. (Note that the morphisms are the inclusions of compact subsets.) Therefore, the canonical map
\[
F(V_{1})\rightarrow \text{holim}_{i} \hspace{0,1cm} M \setminus C_{i}
\]
is a weak equivalence. 
Now we define the compact sets $C_{t,i}:= e_{t} (e_{1}^{-1}(C_{i}))$. Note that $C_{1,i}=C_{i}$. By definition, the inclusion
\[
\left\{C_{0,i}\right\}_{i\geq 0} \rightarrow \left\{C\subset V_{0} \mid C\hspace{0,2cm} \text{compact} \right\}
\]
is homotopy terminal and induces a weak equivalence 
\[
F(V_{0}) \rightarrow \text{holim}_{i} \hspace{0,1cm} M\setminus C_{0,i}
\]
We fix the following notation: 
\[
Y_{i} := \left\{ w:\left[ 0,1 \right]\rightarrow M \mid w(t)\notin M\setminus C_{t,i} \right\}
\]
There are evaluation maps $Y_{i}\rightarrow M\setminus C_{0,i}$ and $Y_{i}\rightarrow M\setminus C_{1,i}$.
Using the isotopy extension theorem \cite[6.5]{siebenmann}, 
it is straightforward to find homotopy inverses. For a comment on the isotopy extension theorem for stratified spaces, see Remark \ref{isotextthm}. We get homotopy eqivalences
\[
 M\setminus C_{0,i} \longleftrightarrow Y_{i} \longleftrightarrow  M\setminus C_{1,i}
\]

Since the evaluation maps are natural, we get weak equivalences
\[
\text{holim}_{i} \hspace{0,1cm} M\setminus C_{0,i} \leftarrow \text{holim}_{i}\hspace{0,1cm} Y_{i} \rightarrow \text{holim}_{i}\hspace{0,1cm} M\setminus C_{1,i}
\]
To summarize, we have shown that the spaces $F(V_{1})$ and $F(V_{0})$ are weakly equivalent. Now we have to argue that the canonical map $F(V_{1})\rightarrow F(V_{0})$ induced by inclusion is a weak equivalence. \\
Let $g : \N \rightarrow \N$ be a monotone injective function such that for every $i\in \N$ and $t\in \left[ 0,1 \right]$, the compact set $C_{t,i}$ is a subset of $C_{1,g(i)}$. We consider the composition
\[
\Psi : \text{holim}_{i} \hspace{0,1cm} M\setminus C_{1,i} \rightarrow \text{holim}_{i}\hspace{0,1cm} M\setminus C_{1,g(i)} \rightarrow \text{holim}_{i}\hspace{0,1cm} M\setminus C_{0,i}
\]
where the first map is induced by the inclusion $\left\{C_{1,g(i)}\right\}_{i} \rightarrow \left\{C_{1,i}\right\}_{i}$ of categories and the second map is induced by the inclusions $C_{0,i}\hookrightarrow C_{1,g(i)}$, $i\in\N$, of spaces. In order to verify that the composition $\Psi$ is a weak equivalence, we consider the following homotopy commutative triangle

\hspace{2cm}\begin{xy}
  \xymatrix{
      \text{holim}_{i} \hspace{0,1cm} Y_{i} \ar[rr]^{\cong} \ar[rd]^{\cong}  &     &  \text{holim}_{i} \hspace{0,1cm} M\setminus C_{1,i} \ar[dl]^{\Psi}  \\
                             &  \text{holim}_{i} \hspace{0,1cm} M\setminus C_{0,i}  &
  }
\end{xy} \\
\vspace{0,2cm}

It does not seem to be trivial that the triangle is homotopy commutative. But by careful inspection, the definition of the homotopy limit provides a homotopy whereby the triangle is homotopy commutative. Using the same argument, we get a homotopy commutative square

\begin{equation*}
\begin{gathered}
\xymatrix{
F(V_{1}) \ar[d]^{\cong} \ar[r] & F(V_{0}) \ar[d]^{\cong} \\
\text{holim}_{i} \hspace{0,1cm} M\setminus C_{1,i}\hspace{0,1cm} \ar[r]^{\Psi} & \hspace{0,1cm}\text{holim}_{i} \hspace{0,1cm} M\setminus C_{0,i}
}
\end{gathered}
\end{equation*}
Since $\Psi$ is a weak equivalence, the canonical map $F(V_{1})\rightarrow F(V_{0})$ is also a weak equivalence.
\end{proof}

\begin{rmk}\label{equivisot}
We need a slight reformulation of a stratified isotopy equivalence. According to Definition \ref{isotopyequ}, an inclusion $i:V_{0}\rightarrow V_{1}$ of open subsets of $K$ is a stratified isotopy equivalence if there is a continuous map $e:V_{1}\rightarrow V_{0}$ such that $e|_{V_{1}\cap S}$ is a smooth embedding from $V_{1}\cap S$ into $V_{0}\cap S$ for all simplices $S$ of $K$ and if there are a stratified isotopy from $i\circ e$ to $id_{V_{1}}$ and a stratified isotopy from $e\circ i$ to $id_{V_{0}}$. The following definition would also be appropriate: We could call an inclusion $i:V_{0}\rightarrow V_{1}$ of open subsets of $K$ a stratified isotopy equivalence if $i$ is stratified isotopic to a homeomorphism $e:V_{0}\rightarrow V_{1}$ such that $e|_{V_{0}\cap S}$ is a diffeomophism from $V_{0}\cap S$ to $V_{1}\cap S$ for all simplices $S$ of $K$. (Note that $S$ is not a manifold, so more precisely we should say: The map $e|_{V_{0}\cap S}$ from $V_{0}\cap S$ to $V_{1}\cap S$ can be extended to a diffeomophism using that $S$ is canonically embedded in an euclidean space.)  \\
Why is the second definition of stratified isotopy equivalences also appropriate? We do not know if these definitions are equivalent, but it is straightforward to verify the following claim: Let $G:\mathcal{O}(K)\rightarrow (Top)$ be a contravariant functor. Then $G$ takes stratified isotopy equivalences as in Definition \ref{isotopyequ} to weak equivalences if and only if $G$ takes stratified isotopy equivalences as in the second definition to weak equivalences.
\end{rmk}

\begin{rmk}\label{isotextthm}
In the proof of the last lemma we can use a continuous version of the isotopy extension theorem for stratified spaces as provided in \cite[6.5]{siebenmann}. It is straightforward to check that in our situation all conditions of \cite[6.5]{siebenmann} are satisfied.
\end{rmk}

\subsection{Proof of the main theorem}

Now we prove Theorem \ref{thetheorem}, i.e. we show that the top horizontal arrow in the commutative diagram
\[
\xymatrix@M=6pt@R=17pt{
M\setminus K
\ar[d] \ar[r] & {\rule{0mm}{9mm}\holimsub{(T,\rho)\in\con(K)} M\setminus V_{K}(T,\rho)} \ar[d] \\
F(K) \ar[r] & {\rule{0mm}{6mm}\holimsub{(T,\rho)\in\con(K)} F(V_{K}(T,\rho))} 
}
\]
%
%
is a weak equivalence.
The left vertical arrow is a weak equivalence because $K$ is a maximal element in the category (poset) of all compact subsets of $K$. The right vertical arrow is a weak equivalence because for every $(T,\rho)\in\con(K)$, the category of all compact subsets of $V_{K}(T,\rho)$ has a directed subcategory which is homotopy terminal. Therefore, we have to show that the bottom horizontal arrow is a weak equivalence. To this end, we will use the good properties of the functor $F$ and manifold calculus (adapted to simplicial complexes). The bottom arrow equals the composition
\[
F(K) \hspace{0,1cm}\rightarrow \holimsub{U\in\cup_{k} \mathcal{O}k(K)} F(U) \hspace{0,1cm}\rightarrow \holimsub{(T,\rho)\in\con(K)} F(V_{K}(T,\rho))
\]
where the first map is the canonical map and the second map is induced by the inclusion of posets
\[
\con (K) \rightarrow  \cup_{k} \mathcal{O}k(K)
\]
given by $(T,\rho) \mapsto V_{K}(T,\rho)$. Therefore, the following two lemmata complete the proof. (The proof of the case with restricted cardinalities follows similar lines.)

\begin{lem}\label{vergleich}
The canonical projection map 
\[
\holimsub{U\in\cup_{k} \mathcal{O}k(K)} F(U) \rightarrow \holimsub{(T,\rho)\in\con(K)} F(V_{K}(T,\rho))
\]
induced by the inclusion $\con (K) \rightarrow  \cup_{k} \mathcal{O}k(K)$ is a weak equivalence.
\end{lem}
\begin{proof}
By \cite[Theorem 6.14]{Dugger}, it remains to show that the canonical map
\[
F(U) \rightarrow \holimsub{(T,\rho)\in\con(K)|_{U}} F(V_{K}(T,\rho))
\]
is a weak equivalence for all $U\in\cup_{k}\mathcal{O}k(K)$. Recall that $\con(K)|_{U}$ is the full subcategory of $\con(K)$ where the objects are all elements $(T,\rho)$ in $\con(K)$ with $V_{K}(T,\rho)\subset U$. For a fixed $U\in\cup_{k}\mathcal{O}k(K)$, we choose an element $(T',\sigma)\in\con(K)|_{U}$ such that the map $F(U)\rightarrow F(V_{k}(T,\sigma))$ is a weak equivalence. 
We set $W:=V_{K}(T',\sigma)$ and consider the following commutative diagram 
\[
\xymatrix@M=6pt@R=17pt{
F(U)
\ar[d] \ar[r] & {\rule{0mm}{9mm}\holimsub{(T,\rho)\in\con(K)|_{U}} F(V_{K}(T,\rho))} \ar[d] \\
F(W) \ar[r] & {\rule{0mm}{6mm}\holimsub{(T,\rho)\in\con(K)|_{W}} F(V_{K}(T,\rho))} 
}
\]
%
%
%
%
The bottom arrow is a weak equivalence because $W$ is a maximal element in $\con(K)|_{W}$. In order to show that
the right vertical arrow is a weak equivalence, we will consider the two appearing homotopy limits as continuous homotopy limits. This is allowed by Remark \ref{resholimcomp}.
Then we compare the two spaces $\con(K)|_{W}$ and $\con(K)|_{U}$. By definition of their topologies, the inclusion $\con(K)|_{W} \rightarrow \con(K)|_{U}$ is a weak equivalence. Similarly, the maps of section spaces $\Gamma_{r}(\Phi)|_{U}\rightarrow \Gamma_{r}(\Phi)|_{W}$ are weak equivalences for all $r\geq 0$. So they induce a weak equivalence of continuous homotopy limits.
\end{proof}

\begin{lem}
The canonical map
\[
F(K) \rightarrow \holimsub{U\in\cup_{k} \mathcal{O}k(K)} F(U)
\]
is a weak equivalence.
\end{lem}
\begin{proof}
Note that we have already shown that $F$ is good (Lemma \ref{Fisgood}). 
Let $P$ be a smooth compact codimension zero subobject of $K$ and let $A_{0},A_{1},...,A_{r}$ be compact codimension zero subobjects of $K\setminus \text{int}(P)$ with relative handle index $q_{A_{i}}$ (relative to $P$). For $T\subset [r]$, we define 
\[
W_{T}:= \text{int} (P \cup \bigcup_{i\in T} A_{i} )
\]
where $\text{int}(...)$ is the interior in $K$. We have to show that the cube
\[
T\mapsto F(W_{T}) \hspace{0,1cm},\hspace{0,3cm} T\subset [r]
\]
is strongly cocartesian and that for every $0\leq i\leq r$, the maps
\[
F\left(W_{[r]}\right) \rightarrow F \left(W_{[r]\setminus \left\{i\right\}}\right)
\] 
are $((m-1)-q_{A_{i}})$-connected where $m$ is the dimension of $M$. Note that $W_{S}$ is the interior of a compact codimension zero subobject of $K$. Therefore,
instead of using the functor $F$, we can work with the cube
\[
T \mapsto G(W_{T}) := M \setminus W_{T}
\]
Why can we use this cube? Because of the special assumption, there is a directed homotopy terminal subcategory in the category of all compact subsets of $W_{T}$. Thus, the canonical map $G(W_{T})\rightarrow F(W_{T})$ is a weak equivalence.    \\
Let $i,j\in [r]$ be two distinct elements. In order to show that the cube induced by $G$ is strongly cocartesian, we need to investigate if the canonical map from the homotopy pushout of
\[
G(W_{[r]\setminus \left\{i\right\}}) \leftarrow G(W_{[r]}) \rightarrow G(W_{[r]\setminus \left\{j\right\}})
\]
to $G(W_{\underline{r}\setminus \left\{i,j\right\}})$ is a weak equivalence. But this can easily be seen. In fact, using the assumptions that all $A_{i}$ are pairwise disjoint, we can find a copy of $G(W_{\underline{r}\setminus \left\{i,j\right\}})$ in the homotopy pushout which is a retract of the homotopy pushout.  
%
Likewise it is not difficult to check that for a fixed $i\in [r]$, the map 
\[
G(W_{[r]}) \rightarrow G(W_{[r]\setminus \left\{i\right\}})
\]
is $(m-q_{A_{i}}-1)$-connected since the target is homotopy equivalent to the source with attached cells of dimension $\geq (m-q_{A_{i}})$. 
\end{proof}

\section{Occupants in the interior of a manifold}\label{boundaryrecovered}

We discuss Theorem \ref{mainvariation} where the homotopy type of the boundary of a compact smooth manifold $M$ (with boundary) is recovered by the homotopy types of the spaces $M\setminus T$ with $T\subset M\setminus\partial M$ finite. To this end, we give the definition of a nice neighborhood of a simplicial complex and discuss first observations and examples. Then we prove the tube lemma (\ref{tubelemma}) which we will need to prove Theorem \ref{mainvariation}.

\subsection{Nice neighborhoods}

Let $M$ be a compact smooth manifold with boundary of dimension $m$.

\begin{defi}\label{niceproj}
Let $K\subset M\setminus \partial M$ be a simplicial complex. We say that $p:M\rightarrow K$ is a \textit{nice projection map} if the following conditions hold:
\begin{enumerate}
    \item $p|_{K} = \text{id}_{K}$
		\item The open set $p^{-1}(V(T,\rho))\setminus\partial M$ is diffeomorphic to $T\times \R^{m}$ for every element $(T,\rho)$ of the configuration category $\con(K)$ of $K$.
		\item The inclusion $\partial M \cup (M\setminus p^{-1}(V)) \rightarrow M\setminus V$ is a weak equivalence for all open sets $V\in\mathcal{O}(K)$. 
\end{enumerate}
\end{defi}

\begin{defi}\label{niceneighb}
We say that $M$ is a \textit{nice neighborhood} of a simplical complex $K\subset M$ if each (closed) simplex of $K$ is smoothly embedded in $M\setminus\partial M$ and if there exists a nice projection map $p:M\rightarrow K$.
\end{defi}

\begin{rmk}
In the second condition of Definition \ref{niceproj} we are just considering the interior of the manifold. But it follows from the definitions that $p^{-1}(V(T,\rho)) \simeq T$, i.e. each component of $p^{-1}(V(T,\rho))$ is contractible.
\end{rmk}

\begin{ex}\label{weaken}
(1) The definition of nice neighborhood weakens the strong condition in \cite[2.1.1]{occupants} in the following sense: Let $L$ be a smooth closed manifold and let $p : M\rightarrow L$ be a smooth disk bundle, i.e. a smooth fiberbundle where each fiber is diffeomorphic to a (closed) disk $D^{r}$ of fixed dimension $r\geq 0$. Then $L$ can be considered as a subset of $M$ by using the zero section of $p$. We can choose a triangulation of $L$ and then $M$ is a nice neighborhood of the triangulation of $L$. \\
(2) We consider the $1$-dimensional simplicial complex $K$ with four vertices $\left\{ a,b,c,d \right\}$ and $1$-simplices $\left\{ \left\{a,b\right\},\left\{a,c\right\},\left\{b,c\right\},\left\{b,d\right\},\left\{c,d\right\} \right\}$, i.e. we have two triangles which coincide in exactly one simplex, namely $\left\{b,c\right\}$. Now it is an easy exercise to build up a compact manifold $M$ of dimension $m = 2$ which is a nice neighborhood of $K$. We ought to consider $M$ as a manifold with with four $0$-handles and five $1$-handles. 
This example can easily be generalized with all dimensions $m\geq 2$ or/and with an one-dimensional simplicial complex which consists of more than two triangles.
\end{ex}

\begin{lem}\label{spaeter}
We assume that $M$ is a nice neighborhood of $K$ and that $\text{dim}(K)+3\leq m$. Let $p:M\rightarrow K$ be a nice projection map. Then the canonical map 
\[
\partial M \rightarrow \holimsub{(T,\rho)\in\con(K)} \partial M \cup (M \setminus p^{-1}(V_{K}(T,\rho )))
\]
is a weak equivalence.
\end{lem}
\begin{proof}
We consider the following five homotopy equivalences:
\begin{align*}
\partial M  &\simeq M\setminus K \\
&\simeq M\setminus (K\cup\partial M) \\
&\simeq \holim_{(T,\rho ) }  M \setminus (V_{K}(T,\rho )\cup \partial M) \\
&\simeq \holim_{(T,\rho ) }  M \setminus V_{K}(T,\rho ) \\
&\simeq \holim_{(T,\rho ) } \partial M \cup (M \setminus p^{-1}(V_{K}(T,\rho )))
\end{align*}
where the three homotopy limits are taken over all $(T,\rho)$ in $\con(K)$.
By definition of nice neighborhoods
, the first equivalence can be verified as well as
the fifth equivalence. By Theorem \ref{thetheorem}, the third map is a weak equivalence. The second and the fourth map are weak equivalences since $M\cong M\setminus \partial M$. 
\end{proof}

\subsection{Tube lemma}

Now we adapt the results of \cite[2.2]{occupants} for a nice projection map. 
Note that for the following lemma we do not have to require that the codimension is at least three. It could also be zero.

\begin{lem}\label{tubelemma}
Let $M$ be a compact, smooth manifold with boundary which is a nice neighborhood of a compact simplicial complex $K$. Let $p: M\rightarrow K$ be a nice projection map (Def. \ref{niceproj}). Then the canonical map
\begin{align}\label{tubemap}
\hocolimsub{(T,\rho)\in\con(K)} C_{n}(p^{-1}(V_{K}(T,\rho))\setminus \partial M) \rightarrow C_{n}(M\setminus \partial M)
\end{align}
is a weak equivalence.
\end{lem}
\begin{proof}
We are going to show that the map is a microfibration with contractible fibers.
Let $T$ be an element of the configuration space $C_{n}(M\setminus \partial M)$. It is obvious that the fiber of $T$ is the classifying space of all $(T,\rho)\in\con(K)$ with $T\in p^{-1}(V_{K}(T,\rho))\setminus\partial M$, i.e. $p(T)\in V_{K}(T,\rho)$. The inclusion of the directed poset 
\[
\left\{ (T,\rho)\in\con(K) \mid \exists n\in\N :\rho(t)=\frac{1}{n} \hspace{0,2cm}\text{for all}\hspace{0,2cm} t\in T \right\}
\]
into the above described poset is a homotopy initial functor. (We consider the posets as categories.) Therefore, the fiber is contractible. \\
Now we verify the lifting condition. We start with an observation: The projection map and the map (\ref{tubemap}) determine an injective, continuous map
\[
\hocolimsub{(T,\rho)\in\con(K)} C_{n}(p^{-1}(V_{K}(T,\rho))\setminus \partial M) \rightarrow \left| N\con(K) \right| \times C_{n}(M\setminus \partial M)
\]
(This map is not an embedding, i.e. homeomorphism onto its image. See also Remark \ref{notmetr}.)  We call this map $g=(g_{1},g_{2})$. \\
Let $Z$ be a compact CW-space. We consider the following diagram
\[
\xymatrix@M=6pt@R=17pt{
Z
\ar[d] \ar[r] & {\rule{0mm}{6mm}\hocolimsub{(T,\rho)\in \con(K)} C_{n}(p^{-1}(V_{K}(T,\rho))\setminus\partial M)}\ar[d] \\
Z \times I  \ar[r] &  C_{n}(M\setminus \partial M)
}
\]
We call the upper horizontal map $f$ and we can consider it as a pair of maps $f=(f_{1},f_{2})$ if we define $f_{i}:= g_{i}\circ f$, $i=1,2$. We call the bottom horizontal map $h$. The right vertical arrow is equal to $g_{2}$. We can define a small lift
\[
H: Z \times \left[ 0, \epsilon \right] \rightarrow \hocolimsub{(T,\rho)\in\con(K)} C_{k}(p^{-1}(V_{K}(T,\rho))\setminus \partial M)
\]
by $H:= (f_{1},h)$. \\
How can we describe the map $H$? Let $z\in Z$ be given. By the formula $H:= (f_{1},h)$, the map 
\[
\left\{z\right\} \times \left[ 0, \epsilon \right] \stackrel{H}{\longrightarrow} \hocolimsub{(T,\rho)\in\con(K)} C_{n}(p^{-1}(V_{K}(T,\rho))\setminus \partial M) \stackrel{g_{1}}{\longrightarrow} N\con (K)
\]
is constant, more precisely: $g_{1} \circ H(\left\{z\right\} \times \left[ 0, \epsilon \right] )=\left\{f_{1}(z)\right\}$. \\
How can we find an $\epsilon > 0$ such that $H$ is well-defined? Let $S$ be an $r$-simplex of $\left|N\con (K)\right|$, let $E$ be the corresponding open simplex and let $(T_{0},\rho_{0}) \leq ... \leq (T_{r},\rho_{r})$ be the corresponding element in $N_{r}\con(K)$. We define 
\begin{align*}
Z_{S}&:= f_{1}^{-1}(S) = f^{-1}(g_{1}^{-1}(S)) \subset Z \\
Z_{E}&:= f_{1}^{-1}(E) = f^{-1}(g_{1}^{-1}(E)) \subset Z
\end{align*}
We take a close view at the map
\[
f_{2}|_{Z_{S}}: Z_{S} \stackrel{f}{\longrightarrow} \hocolimsub{(T,\rho)\in\con(K)} C_{n}(p^{-1}(V_{K}(S,\rho))\setminus \partial M) \stackrel{g_{2}}{\longrightarrow} C_{n}(M\setminus \partial M)
\]
First, we note that $f_{2}(Z_{E}) \subset C_{n}(p^{-1}(V_{K}(T_{0},\rho_{0}))\setminus \partial M)$ by definition.
By definition (of nice neighborhood), $(p^{-1}(V_{K}(S_{j},\rho_{j}))\setminus \partial M)$ is a special open set for every $0\leq j\leq r$. In the spirit of Remark \ref{notmetr}, we conclude that $f_{2}(Z_{S})$ is also a subset of $C_{n}(p^{-1}(V_{K}(T_{0},\rho_{0}))\setminus \partial M)$. (For an easier example of this argument, see \cite[2.2.1]{occupants}.) Since $f_{2}(Z_{S}) = h (Z_{S} \times \left\{0\right\})$ is compact, there is an $\epsilon_{S} > 0$ with
\[
h\left( Z_{S} \times \left[ 0, \epsilon_{S} \right]\right) \subset C_{n}(p^{-1}(V_{K}(T_{0},\rho_{0}))\setminus \partial M)
\]
The image of $Z$ is contained in a finite union of open cells of $\left|N\con(K)\right|$. Therefore, there is a finite number of simplices $S$ such that $Z_{S}$ is nonempty. We can define $\epsilon$ to be the minimum of all $\epsilon_{S}$ where the minimum ranges over all simplices $S$ such that $Z_{S}$ is nonempty. 
\end{proof}

\begin{rmk}\label{notmetr}
Let $U\in\R^{n}$ be a bounded open subset. Then the mapping cylinder of the inclusion $U\rightarrow\R^{n}$ is not homeomorphic to a subspace of $\R^{m+1}$. The quotient topology equips the mapping cylinder with a different structure. In fact, it is not metrizable. 
\end{rmk}

\begin{cor}
The canonical map
\[
\hocolimsub{(T,\rho)\in\con(K)} N_{0}\con(p^{-1}(V_{K}(T,\rho))\setminus\partial M) \rightarrow N_{0}\con(M\setminus\partial M)
\]
determined by the inclusions is a weak equivalence.
\end{cor}
\begin{proof}
We define $\con(U)$ to be the full subcategory of $\con(K)$ with all objects $(T,\rho)$ such that the (compact) closure of $V_{K}(T,\rho)$ is a subset of $U$. 
There is a commutative square
\[
\xymatrix@M=6pt@R=17pt{
{\rule{0mm}{6mm}\hocolimsub{(T,\rho)\in \con(K)} N_{0}\hspace{1mm}\con(p^{-1}(V_{K}(T,\rho))\setminus\partial M)}
\ar[d] \ar[r] & N_{0}\hspace{1mm}\con(M\setminus\partial M) \ar[d] \\
{\rule{0mm}{6mm}\hocolimsub{(T,\rho)\in \con(K)} \coprod_{n} C_{n}(p^{-1}(V_{K}(T,\rho))\setminus\partial M)}  \ar[r] &  \coprod_{n} C_{n}(M\setminus\partial M)
}
\]
where the vertical arrows are weak equivalences (the left one is induced by a natural transformation). Note: Here we need the assumption that the (compact) closure of $V_{K}(T,\rho)$ is a subset of $U$. If we defined $\con(U)$ to be the full subcategory of $\con(K)$ with all objects $(T,\rho)$ such that $V_{K}(T,\rho)$ is a subset of $U$, then in general the left vertical map would \textbf{not} be a weak equivalence. \\
Therefore, we only have to verify that the bottom map is a weak equivalence. But this follows from the fact that the homotopy colimit commutes with disjoint union. 
\end{proof}

\begin{cor}\label{tube}
For every $r\geq 0$, the canonical map
\[
\hocolimsub{(T,\rho)\in\con(K)} N_{r}\con(p^{-1}(V_{K}(T,\rho))\setminus\partial M) \rightarrow N_{r}\con(M\setminus\partial M)
\]
induced by the inclusions is a weak equivalence.
\end{cor}
\begin{proof}
We consider the following commutative square:
\[
\xymatrix@M=6pt@R=17pt{
{\rule{0mm}{6mm}\hocolimsub{(T,\rho)\in \con(K)} N_{r}\hspace{1mm}\con(p^{-1}(V_{K}(T,\rho))\setminus\partial M)}
\ar[d] \ar[r] & N_{r}\hspace{1mm}\con(M\setminus\partial M) \ar[d] \\
{\rule{0mm}{6mm}\hocolimsub{(T,\rho)\in \con(K)} N_{0}\hspace{1mm}\con(p^{-1}(V_{K}(T,\rho))\setminus\partial M)}  \ar[r] &  N_{0}\hspace{1mm}\con(M\setminus\partial M)
}
\]
Here the vertical arrows are given by the ultimative target operator and the horizontal arrows are the canonical maps induced by the inclusions. We can check that this is a (strict) pullback square and that the right vertical arrow is a fibration. Since $(Top)$ is a proper model category \cite[13.1.11]{Hirschhorn} and the bottom arrow is weak equivalence, we conclude that the upper arrow is also a weak equivalence.
\end{proof}

\subsection{Boundary recovered}\label{formulation2}

Let $M$ be a manifold with boundary $\partial M$. We recover the homotopy type of $\partial M$ from the homotopy types of the spaces $M\setminus T$ where $T$ is a finite subset of $M\setminus \partial M$. Again, we need to allow thickenings of the finite subsets $T$ and inclusions between them. We recall that for each object $(T,\rho)$ in the configuration category $\con(M\setminus \partial M)$ of $M\setminus\partial M$, there is a corresponding open subset $V_{M\setminus\partial M}(T,\rho)$ in $M\setminus\partial M$. For simplicity, we will write $V(T,\rho)$ instead of $V_{M\setminus\partial M}(T,\rho)$. 
We can define a contravariant functor $\psi$ from $\con(M\setminus \partial M)$ to the category of topological spaces by $\psi ((T,\rho)) := M\setminus V(T,\rho)$. We get a canonical map
\begin{align}\label{boundary}
\partial M \rightarrow \holimsub{(T,\rho)\in \con(M\setminus\partial M)} M\setminus V(T,\rho) 
\end{align}
induced by the inclusions
$\partial M \rightarrow M\setminus V(T,\rho)$. We can ask if this map is a weak equivalence. There is also a variant with restricted cardinalities. Let $\con_{\leq n}(M\setminus \partial M)$ be the full subcategory of $\con(M\setminus \partial M)$ where the objects are all pairs $(T,\rho)\in  \con(M\setminus \partial M)$ with $\left|T\right|\leq n$. Again, we get a canonical map
\begin{align}\label{bres}
\partial M \rightarrow \holimsub{(T,\rho)\in \con_{\leq n}(M\setminus\partial M)} M\setminus V(T,\rho)
\end{align}
induced by inclusions. We can ask wether this map is highly connected and wether there is a lower bound for the connectivity. The following theorem where we use again the notations $\kappa:=\text{dim}(K)$ and $m:=\text{dim}(M)$ answers these questions.

\begin{thm}\label{generalization}
The canonical map (\ref{boundary}) is a weak equivalence if the following condition holds: There is a compact simplicial complex $K\subset M$ of dimension $\kappa$ with $\kappa+3\leq m$ such that $M$ is a \textit{nice neighborhood} (Def. \ref{niceneighb}) of $K$. In this case, the canonical map (\ref{bres}) is $1+(n+1)(m-\kappa-2)$-connected.
\end{thm}

\begin{rmk}
In (\ref{boundary}) and (\ref{bres}), the discrete (or ordinary) homotopy limit can be replaced by the continuous homotopy limit without changing the (weak) homotopy type. This can be justified with arguments which are provided in \cite[1.2]{occupants} (and in section 2.2).
\end{rmk}

\begin{rmk}
This theorem is a generalization of \cite[Thm. 2.1.1]{occupants}, compare Example \ref{weaken} (2). It can be applied in the proof of \cite[Thm. 5.2.1 and 5.3.1]{confcat} whereby we get a weaker condition in these theorems (this will extensively be studied in section \ref{aplli}). 
\end{rmk}
In order to prove that (\ref{boundary}) is a weak equivalence, we consider the following diagram where all arrows are the canonical maps and $p:M\rightarrow K$ is a nice projection map:
\[
\xymatrix@M=6pt@R=17pt{
\partial M
\ar[d] \ar[r] & {\rule{0mm}{6mm}\holimsub{(T,\rho)\in \con(K)} \partial M\cup\big(M\setminus p^{-1}(V_K(T,\rho))\big)}\ar[d] \\
{\rule{0mm}{6mm}\holimsub{(T',\sigma)\in \con(M\setminus\partial M)}\psi(T',\sigma)}  \ar[r] &  {\rule{0mm}{9mm}\holimsub{(T,\rho)\in \con(K)} \!\!\!\!\!\!
\holimsub{\twosub{(T',\sigma)\in \con (M\setminus\partial M)}{p(V(T',\sigma))\subset V_{K}(T,\rho)}} \psi(T',\sigma)} 
}
\]
%
%
It commutes because both compositions factorize through the ordinary limit. In Lemma \ref{spaeter} we have already shown that the upper horizontal arrow is a weak equivalence. Therefore, the first part of the theorem follows from the next two lemmata.

\begin{lem}\label{rightarrow}
The right vertical arrow is a weak equivalence.
\end{lem}
\begin{proof}
Let $(T,\rho)\in \con(K)$ be fixed. Since the map under investigation is induced by a natural transformation, it suffices to show that the map 
\begin{align*}
\partial M \cup (M \setminus p^{-1}(V_{K}(T,\rho ))) = M\setminus U \hspace{0,2cm} \rightarrow \hspace{0,1cm}
 \holimsub{(T',\sigma )\in\con(U)} \psi (T',\sigma )
\end{align*}
is a weak equivalence where for simplicity, $U$ is defined to be the open set
\[
U := p^{-1}(V_{K}(T,\rho ))\setminus\partial M \hspace{0,3cm} \subset \hspace{0,3cm} M\setminus \partial M
\]
Note that by definition, the open set $U$ is diffeomorphic to $T\times \R^{m}$. We consider the following compostion of maps:
\begin{align*}
M\setminus U \hspace{0,2cm}
\rightarrow \hspace{0,1cm} \holimsub{(T',\sigma)\in\con(U)} \psi (T',\sigma ) \hspace{0,2cm}
\rightarrow \hspace{0,1cm}  \holimsub{(T',\sigma)\in\con(U)} F (V_{K}(T',\sigma )) 
\end{align*}
where $F$ is the functor from the category $\mathcal{O}(U)$ of open subsets of $U$ to topological spaces given by $F(W) := \holim_{C\subset W }\hspace{0,1cm} M\setminus C$ where $C$ runs through the compact subsets of $W$. Note that the category of all compact subsets of $V_{K}(T',\sigma)$ has a directed subcategory which is homotopy terminal. Therefore, the canonical map $\psi (T',\sigma ) \rightarrow  F (V_{K}(T',\sigma )) $ is a weak equivalence for every $(T',\sigma)\in\con(U)$. Using the homotopy invariance of the homotopy limit, the second map is a weak equivalence. So, in order to prove that the first map is a weak equivalence, we have to show that the composition is a weak equivalence. To this end, we consider another composition
\begin{align*}
M\setminus U \hspace{0,1cm}
\rightarrow \hspace{0,1cm} F(U) \hspace{0,1cm}
\rightarrow    \holimsub{W\in\cup_{k}\mathcal{O}k(U)} F (W) \hspace{0,1cm}
\rightarrow   \holimsub{(T',\sigma)\in\con(U)} F (V_{K}(T',\sigma ))
\end{align*}
First of all, we note that the two compositions give the same map since both compositions factorize through the ordinary limit. The first map in this composition is a weak equivalence because the category of all compact subsets of $U$ has a directed subcategory which is homotopy terminal. The third map is a weak equivalence by an arguments which we have seen in Lemma \ref{vergleich}. The second map is a weak equivalence because the open set $U$ is a maximal element in $\cup_{k}\mathcal{O}k(U)$. 
\end{proof}

\begin{lem}
The bottom horizontal arrow is a weak equivalence.
\end{lem}
\begin{proof}
If replace the homotopy limit by the continuous homotopy limit, the source is the totalization of the cosimplicial space $\left[r\right]\mapsto \Gamma_{r}(\Psi)$ where $\Gamma_{r}(\Psi)$ is the space of all sections from $N_{r}\mathcal{P}(M\setminus\partial M)$ to $E^{!}(\Psi)$. (All notations are introduced in section \ref{conthomotlim}.) If replace the second homotopy limit in the target by the continuous homotopy limit (compare Remark \ref{resholimcomp}), the target is isomorphic to the totalization of the cosimplicial space $\left[r\right]\mapsto \tilde{\Gamma}_{r}(\Psi)$ where $\tilde{\Gamma}_{r}(\Psi)$ is the space of all sections from 
\[
\hocolimsub{(T,\rho)\in\con(K)} N_{r}\hspace{0,1cm}\con(p^{-1}(V_{K}(T,\rho))\setminus\partial M)
\]
to $E^{!}(\Psi)$. The bottom horizontal arrow in the above diagram is induced by composition with the map in Corollary \ref{tube}:
\[
\hocolimsub{(T,\rho)\in\con(K)} N_{r}\hspace{0,1cm}\con(p^{-1}(V_{K}(T,\rho))\setminus\partial M)\rightarrow N_{r}\hspace{0,1cm}\con(M\setminus\partial M) \rightarrow E^{!}(\Psi)
\]
Using Corollary \ref{tube}, this map is a weak equivalence.
\end{proof}

Now we investigate the case with restricted cardinalities. To this end, we fix $n\geq 0$. 
Let $j$ be an integer with $0\leq j\leq n$ be given. There is the following modification of the tube lemma (\ref{tubelemma}). The canonical map
\[
\hocolimsub{(T,\rho)\in\con_{\leq n}(K)} C_{j}(p^{-1}(V_{K}(T,\rho))\setminus\partial M) \rightarrow C_{j}(M\setminus\partial M)
\]
is a weak equivalence. The proof is the same: The projection map is a microfibration with contractible fibers. Why do we need that $j\leq n$? In the proof of Lemma \ref{tubelemma} we introduced a homotopy initial subposet, in order to show that the fibers are contractible. In the restricted case, this poset is defined if and only if $j\leq n$. \\
Using this observation, the proof of the restricted case follows similar lines. In particular, there is a commutative diagram
\[
\xymatrix@M=3pt@R=17pt{
\partial M
\ar[d] \ar[r] & {\rule{0mm}{6mm}\holimsub{(T,\rho)\in \con_{\leq n}(K)} \partial M\cup\big(M\setminus p^{-1}(V_K(T,\rho))\big)}\ar[d] \\
{\rule{0mm}{6mm}\holimsub{(T',\sigma)\in \con_{\leq n}(M_{-})}\psi(T',\sigma)}  \ar[r] &  {\rule{0mm}{9mm}\holimsub{(T,\rho)\in \con_{\leq n}(K)} \!\!\!\!\!\!
\holimsub{\twosub{(T',\sigma)\in \con_{\leq n} (M_{-})}{p(V(T',\sigma))\subset V_{K}(T,\rho)}} \psi(T',\sigma)} 
}
\]
where $M_{-} := M\setminus \partial M$.
%
%
By Theorem \ref{thetheorem} (and Lemma \ref{spaeter}), the top horizontal map is $1+(n+1)(m-k-2)$-connected. Using a modification of Corollary \ref{tube}, the bottom horizontal arrow is a weak equivalence. In order to justify that the right vertical arrow is a weak equivalence, we can use arguments which we have seen in Lemma \ref{rightarrow}.

\section{Homotopy automorphisms}\label{aplli}

Let $M$ be a smooth, compact manifold with boundary.

\begin{defi}\label{holink}
We define the \textit{homotopy link} $\text{holink}(M / \partial M,*)$ of the base point in $M / \partial M$ to be the space of paths $\gamma:\left[ 0,1 \right] \rightarrow M / \partial M$ which satisfy the condition $\gamma^{-1}(\left\{*\right\}) = \left\{ 0 \right\}$. The topology is the compact-open topology. 
We define the map 
\[
q_{M} : \text{holink}(M / \partial M,*) \rightarrow M \setminus \partial M
\]
by $\gamma\mapsto \gamma(1)$.
\end{defi}

\begin{rmk}
It is well-known that the map $q_{M}$ is a good homotopical substitute for the inclusion map $\partial M\hookrightarrow M$: If we define $Z_{M}$ to be the space of paths $\gamma:\left[ 0,1 \right] \rightarrow M$ which satisfy the condition $\gamma^{-1}(\partial M) = \left\{ 0 \right\}$ (with the compact-open topology), we get a homotopy commutative diagram 
\[
\xymatrix@M=6pt@C=12pt{
\holink(M/\partial M,\star) \ar[rr]^-{q_M} && M\setminus\partial M \ar[d]^-\simeq \\
Z_M \ar[u]^-\simeq \ar[r]^-\simeq & \partial M \ar[r] ^-{\textup{incl.}} & M
}
\]
Let $\homeo (M)$ be the homeomorphism group of $M$. Evidently, there is a canonical action of $\homeo (M)$ on the complete diagram. This action extends to an action of the homeomorphism group $\homeo (M\setminus\partial M)$ on $q_{M}$. But unfortunately, the action does not extend to an action of the homeomorphism group $\homeo (M\setminus\partial M)$ on the inclusion map $\partial M\hookrightarrow M$. We are interested in this extension. That is why we introduced the homotopical substitute $q_{M}$.
\end{rmk}

\begin{defi}
Let $c$ be an object in a model category $\mathcal{C}$. We define $\haut (c)$ to be the \textit{space of derived homotopy automorphisms} of $c$ in $\mathcal{C}$, i.e. $\haut (c)$ is the union of the homotopy invertible path components of the derived mapping space $\R \text{map} (c,c)$. With composition $\haut (c)$ is a grouplike topological or simplicial monoid. (For a suitable definition of simplicial mapping spaces, we follow \cite{DwyerKan}.)
\end{defi}

Note that the map $q_{M}$ can be regarded as a functor from the totally ordered set $\left\{ 0,1 \right\}$ to the category of topological spaces. The category of such functors has well-known standard model category structures. If we choose one of them, we can study the space of derived homotopy automorphisms $\haut (q_{M})$ of $q_{M}$. In particular, since $\homeo (M\setminus\partial M)$ acts on $q_{M}$, each homeomorphism of $M\setminus\partial M$ determines a (derived) homotopy automorphism of $q_{M}$. Therefore, we get a map
\begin{align}\label{Bmap1}
B \homeo (M\setminus \partial M) \rightarrow B \haut (q_{M})
\end{align}
of classifying spaces. \\

Let $\Fin$ be the category of finite sets and maps between them. The nerve $N\Fin$ is a simplicial set. We introduced the Riemannian model of the configuration category $\con(M\setminus\partial M)$. The nerve of this category is a simplicial space over $N\Fin$. 

\begin{defi}
Let $X$ be a simplicial space over $N\Fin$. We define $\haut_{N\Fin} (X)$ to be the \textit{space of derived homotopy automorphisms} of $X$ over $N\Fin$, i.e. $\haut (X)$ is the union of the homotopy invertible path components of the derived mapping space $\R \text{map}_{N\Fin} (X,X)$ of $X$ over $N\Fin$. (If an introduction to derived mapping spaces of simplicial spaces is needed, we refer the reader to \cite[\S 3]{confcat}). With composition $\haut_{N\Fin} (X)$ is a grouplike topological or simplicial monoid.
\end{defi}

If we use the particle model \cite[3.1]{pedromichael}, \cite[\S 1]{confcat} of the configuration category $\con(M\setminus\partial M)$, it is easy to see that each homeomorphism of $M\setminus\partial M$ determines a (derived) homotopy automorphism of the nerve of $\con(M\setminus\partial M)$ over $N\Fin$. \\
\textit{Particle model:} In this model the space of objects of the configuration category $\con(M\setminus\partial M)$ is 
\[
\coprod_{k\geq 0} \emb(\underline{k},M\setminus\partial M) 
\]
A morphism from $f\in\emb(\underline{k},M\setminus\partial M)$ to $g\in\emb(\underline{l},M\setminus\partial M)$ is a map $v:\underline{k}\rightarrow \underline{l}$ and a homotopy \[
(\gamma_{t})_{t\in\left[ 0,a \right]}:\underline{k} \rightarrow M\setminus\partial M
\]
from $f$ to $gv$ which satisfies the \textit{stickiness condition}: if $\gamma_{s}(b_{1}) = \gamma_{s}(b_{2})$ for $s\in\left[ 0, a \right]$ and $b_{1},b_{2}\in \underline{k}$, then $\gamma_{t}(b_{1}) = \gamma_{t}(b_{2})$ for all $t\in \left[ s,a \right]$. Therefore, the space of morphisms of the configuration category $\con(M\setminus\partial M)$ in the particle model is 
\[
\coprod_{k,l\geq 0, \hspace{0,1cm}v:\underline{k}\rightarrow \underline{l}} \Lambda(v)
\]
Here $\Lambda(v)$ is the space of all triples $(f,g,\gamma)$ where $f\in\emb(\underline{k},M\setminus\partial M)$, $g\in\emb(\underline{l},M\setminus\partial M)$ and $\gamma$ is a homotopy from $f$ to $gv$ which satisfies the stickiness condition. The Riemannian model of the configuration category and the particle model are equivalent \cite[3.2]{pedromichael}. \\
Using the particle model of the configuration category $\con(M\setminus\partial M)$, there is an inclusion map of toplogical grouplike monoids from $\homeo(M\setminus\partial M)$ to $\haut_{N\Fin} (\con (M\setminus\partial M))$. We get a map of classifying spaces
\begin{align}\label{Bmap2}
B \homeo(M\setminus\partial M) \rightarrow B \haut_{N\Fin} (\con (M\setminus\partial M))
\end{align}
Now we can ask wether the map (\ref{Bmap1}) has a factorization through the map (\ref{Bmap2}).

\begin{thm}\label{apl}
We assume that the following condition holds: There is a compact simplicial complex $K\subset M$ of dimension $\kappa$ with $\kappa+3\leq m$ such that $M$ is a \textit{nice neighborhood} (Def. \ref{niceneighb}) of $K$.
Then the broken arrow in the homotopy commutative diagram
\[
\xymatrix@C=35pt@M=8pt@R=20pt{
B\homeo(M\setminus\partial M) \ar[r]^-{(\ref{Bmap1})} & B\haut(q_{M}) \\
B\homeo(M\setminus\partial M) \ar@{=}[u] \ar[r]^-{(\ref{Bmap2})} & B\haut_{N\Fin}(\con(M\setminus\partial M)) \ar@{..>}[u]
}
\] 
can be supplied.
\end{thm}

Using Theorem \ref{boundary}, the proof is equal to that of \cite[Thm. 5.2.1]{confcat}.
There is also a variant with restricted cardinalities. Following \cite[5.3]{confcat}, we need a Postnikov decomposition of the map $q_{M}$. It is well-known that for each integer $a\geq 0$, there is a decomposition
\[
\partial M \rightarrow \wp_a \partial M \rightarrow M
\]
of the inclusion map $\partial M \hookrightarrow M$ such that the homotopy groups of $\wp_a \partial M$ are zero in dimension $\geq a+2$ and equal to the homotopy groups of $\partial M$ in dimension $\leq a+1$. ($\wp_a \partial M$ is obtained from $\partial M$, as a space over $M$,
by killing the relative homotopy groups of $\partial M\to M$ in dimensions $\ge a+2$.) By analogy with this construction, there is a decomposition
\[
\text{holink}(M / \partial M,*) \rightarrow  \wp_a (q_{M}) \rightarrow M \setminus \partial M 
\]
of the map $q_{M}$ where $\wp_a (q_{M})$ has the same properties as $\wp_a \partial M$.

\begin{thm} 
We assume that the following condition holds: There is a compact simplicial complex $K\subset M$ of dimension $\kappa$ with $\kappa+3\leq m$ such that $M$ is a \textit{nice neighborhood} (Def. \ref{niceneighb}) of $K$.
Then the broken arrow in the homotopy commutative diagram
\[
\xymatrix@C=35pt@M=8pt@R=20pt{
B\homeo(M\setminus\partial M) \ar[r]^-{\textup{action}} & B\haut(\,\wp_{(j+1)(m-\kappa -2)}(q_{M})) \\
B\homeo(M\setminus\partial M) \ar@{=}[u] \ar[r]^-{\textup{action}} & B\haut_{N\Fin}\big(\con_{\leq j}(M\setminus\partial M)\big) \ar@{..>}[u]
}
\]
can be supplied. Here the two action maps are the maps (\ref{Bmap1}) and (\ref{Bmap2}) applied to the restricted case.
\end{thm}

\end{document}